\documentclass[12pt,reqno]{article}

\usepackage[usenames]{color}
\usepackage{amssymb}
\usepackage{graphicx}
\usepackage{amscd}
\usepackage{tikz}
\usetikzlibrary{automata,positioning}

\def\Enn{{\mathbb{N}}}
\def\Zee{{\mathbb{Z}}}

\usepackage[colorlinks=true,
linkcolor=webgreen,
filecolor=webbrown,
citecolor=webgreen]{hyperref}

\definecolor{webgreen}{rgb}{0,.5,0}
\definecolor{webbrown}{rgb}{.6,0,0}

\usepackage{color}
\usepackage{fullpage}
\usepackage{float}

\usepackage{graphics,amsmath,amssymb}
\usepackage{amsthm}
\usepackage{amsfonts}
\usepackage{latexsym}
\usepackage{epsf}

\def\Zee{\mathbb{Z}}

\def\modd#1 #2{#1\ \mbox{\rm (mod}\ #2\mbox{\rm )}}

\title{A Frameless $2$-Coloring of the Plane Lattice}

\author{Craig S. Kaplan and Jeffrey Shallit\\
School of Computer Science\\
University of Waterloo\\
Waterloo, ON  N2L 3G1 \\
Canada\\
\href{mailto:csk@uwaterloo.ca}{\tt csk@uwaterloo.ca} \\
\href{mailto:shallit@uwaterloo.ca}{\tt shallit@uwaterloo.ca} \\
}

\date{}

\begin{document}

\maketitle

\theoremstyle{plain}
\newtheorem{theorem}{Theorem}
\newtheorem{corollary}[theorem]{Corollary}
\newtheorem{lemma}[theorem]{Lemma}
\newtheorem{proposition}[theorem]{Proposition}

\theoremstyle{definition}
\newtheorem{definition}[theorem]{Definition}
\newtheorem{example}[theorem]{Example}
\newtheorem{conjecture}[theorem]{Conjecture}
\newtheorem{thm}{}

\theoremstyle{remark}
\newtheorem{remark}[theorem]{Remark}

\begin{abstract}
A {\it picture frame} in two dimensions is a rectangular array of
symbols, with at least two rows and columns,
where the first and last rows are identical, and the
first and last columns are identical.  If a coloring of the plane
lattice has no picture frames, we call it {\it frameless}.  In this
note we show how to create a simple $2$-coloring of the plane lattice
that is frameless.
\end{abstract}

\section{Introduction}


A classic problem from combinatorics on words is the avoidance of
overlaps.   An {\it overlap} is a word of the form $axaxa$, where
$a$ is a single symbol and $x$ is a (possibly empty) word. The
English word {\tt alfalfa}, for example, forms an overlap with $a={\tt a}$ and $x={\tt lf}$.
The Norwegian mathematician Axel
Thue (1863--1922) proved that there exists a (one-sided) infinite word
over
a binary alphabet that contains no overlaps.   His classic
paper, in German, was published in 1912 in an obscure Norwegian
journal \cite{Thue:1912}.   Thankfully, an English translation is available
\cite{Berstel:1995}.   

Today the word that Thue produced is called the {\it Thue-Morse word},
and is often abbreviated ${\bf t}$.  Its first few terms are
$${\bf t} = t(0) t(1) t(2) \cdots = {\tt 0110100110010110 } \cdots.$$
The Thue-Morse word $\bf t$ has many equivalent definitions
\cite{Allouche&Shallit:1999}, but the following three are the
most important:
\begin{itemize}
\item $t(n)$ is the number of $1$ bits, taken modulo $2$, in the binary
representation of $n$.
\item $\bf t$ is the limit of $\mu^n(0)$, where $\mu$ is the morphism
defined by $\mu(0) = 01$ and $\mu(1) = 10$, and the exponent on 
$\mu$ denotes $n$-fold composition of morphisms.   This limit is well-defined
since $\mu(0)$ starts with $0$, implying that 
$\mu^n (0)$ is a prefix of $\mu^{n+1} (0)$ for all $n$, and is written
as $\mu^\omega(0)$.
\item There is a finite automaton of two states, which given $n$ expressed
in base $2$ as an input, computes $t(n)$.  This means that
$\bf t$ is a {\it $2$-automatic sequence} \cite{Allouche&Shallit:2003,Rowland:2015}.  
\end{itemize}

Thue's original proof was neither very difficult nor very simple.  Today,
however, however, theorem-proving software can easily
prove his result in less than
a second.   The idea is as follows:   we write a predicate in first-order
logic that asserts that $\bf t$ is overlap-free.   We then type this
predicate into {\tt Walnut}, software written by Hamoon Mousavi that
can prove or disprove any suitable predicate concerning automatic
sequences \cite{Mousavi:2016}, and read the result.

If $\bf t$ has an overlap $axaxa$, then there exist integers
$i \geq 0$ and $n = |ax| \geq 1$ such that $t(i+j) = t(i+n+j)$ for
$0 \leq j \leq n$.   So the nonexistence of overlaps is specified by
the predicate
$$ \forall i, n \ (i \geq 0 \ \wedge \ n \geq 1) \implies
(\exists j \ j \geq 0 \ \wedge \ j \leq n  \ \wedge \ t(i+j) \not= t(i+n+j) ) .$$
A Walnut variable {\tt T} encodes the Thue-Morse word. We can
translate the predicate above into a {\tt Walnut} query named {\tt tmpred}
as follows: \\
\centerline{\tt eval tmpred "A i,n (n >= 1) => (Ej j<=n \& T[i+j] != T[i+n+j])":} 
{\tt Walnut} evaluates this predicate as {\tt true}, which we can take
as proof that the Thue-Morse word is overlap-free! 
(We don't have to include $i \geq 0$ and $j \geq 0$ in our {\tt Walnut}
predicate because the domain of variables in its
predicates is assumed to be $\Enn$ by default.)

Today, the field of pattern avoidance in words is extremely broad
and dynamic, with many generalizations:  avoidance in
circular words \cite{Currie:2005}, two-dimensional
words \cite{Bean&Ehrenfeucht&McNulty:1979}, graphs \cite{Grytczuk:2007}, 
and so forth.   

Thue himself proved a generalization
of his result to ``two-sided'' infinite words.   These are maps
from $\Zee$ to a finite alphabet $\Delta$.   We can turn a ``right-infinite''
infinite word into a ``left-infinite'' word with the reversal operator,
which is denoted by an exponent of $R$.   And we can concatenate
a left-infinite word to a right-infinite word to produce a two-sided
infinite word.  Thue proved
\begin{theorem}
The two-sided infinite word ${\bf u} := {\bf t}^R {\bf t} = \cdots t(3) t(2) t(1) t(0) t(0) t(1)  t(2) t(3) \cdots$ is overlap-free.
\end{theorem}

\begin{proof}
Recall that a morphism $h$ is a map from words to words that obeys
the rule $h(xy) = h(x) h(y)$ for all words $x, y$.

Let $\nu$ be the morphism defined by $\nu(0) = 10$ and $\nu(1) = 01$.
An easy induction shows that $\nu^{2n+1} (0)$ gives the last $2^{2n+1}$ bits
of ${\bf t}^R$.   Similarly, another easy induction shows that
$\nu^{2n+1} (0) = \mu^{2n+1} (1)$.   So the ``central'' $2^{2n+2}$ bits
of ${\bf t}^R {\bf t}$, that is, the word
$t(2^{2n+1} - 1) t(2^{2n+1}-2) \cdots t(1) t(0) t(0) t(1) \cdots
t(2^{2n+1} - 2) t(2^{2n+1} - 1)$, equals $\nu^{2n+1} (0) \mu^{2n+1} (0)$.
But $$\nu^{2n+1} (0) \mu^{2n+1} (0) = 
\mu^{2n+1} (1) \mu^{2n+1} (0) = \mu^{2n+2} (1),$$ which since
$\mu^{2n+3} (0) = \mu^{2n+2} (01)$, appears in $\bf t$ and hence
is overlap-free.
\end{proof}

Thue's two-sided infinite word $\bf u$
can alternatively be defined simply by extending the
domain of $t$ as follows:  if $n$ is negative, then set
\begin{equation}
t(n) := t(-1-n).  \label{ext}
\end{equation}
Then ${\bf u} = \cdots t(-2) t(-1) t(0) t(1) t(2) \cdots$.

In this note we give a surprising example of pattern avoidance
in the plane lattice $\Zee \times \Zee$.   Although our result can be
stated in the language of two-dimensional words, here
it seems somehow
more natural to use the language of colorings instead.   A
{\it coloring} of the plane lattice is a map
$f: \Zee \times \Zee$ into a finite alphabet $\Delta$, where
the elements of $\Delta$ are called {\it colors}.   If the number
of elements of $\Delta$ is $k$, we call $f$ a {\it $k$-coloring}.

The particular kind of pattern we avoid is called a 
``picture frame'', or just ``frame'' for short.  A {\it picture frame}
is a finite rectangular array of points in the plane lattice, with
at least two rows and columns, such that the first row equals the 
last row, and the first column equals the last column.   The
following diagram illustrates part of a plane coloring having
a picture frame in it.  In this case the alphabet $\Delta$ consists
of English letters, and the highlighted frame is bordered by the
horizontal word {\tt ERMINE} and the vertical word {\tt EWE}.

\begin{figure}[H]
\begin{center}
\includegraphics[width=4in]{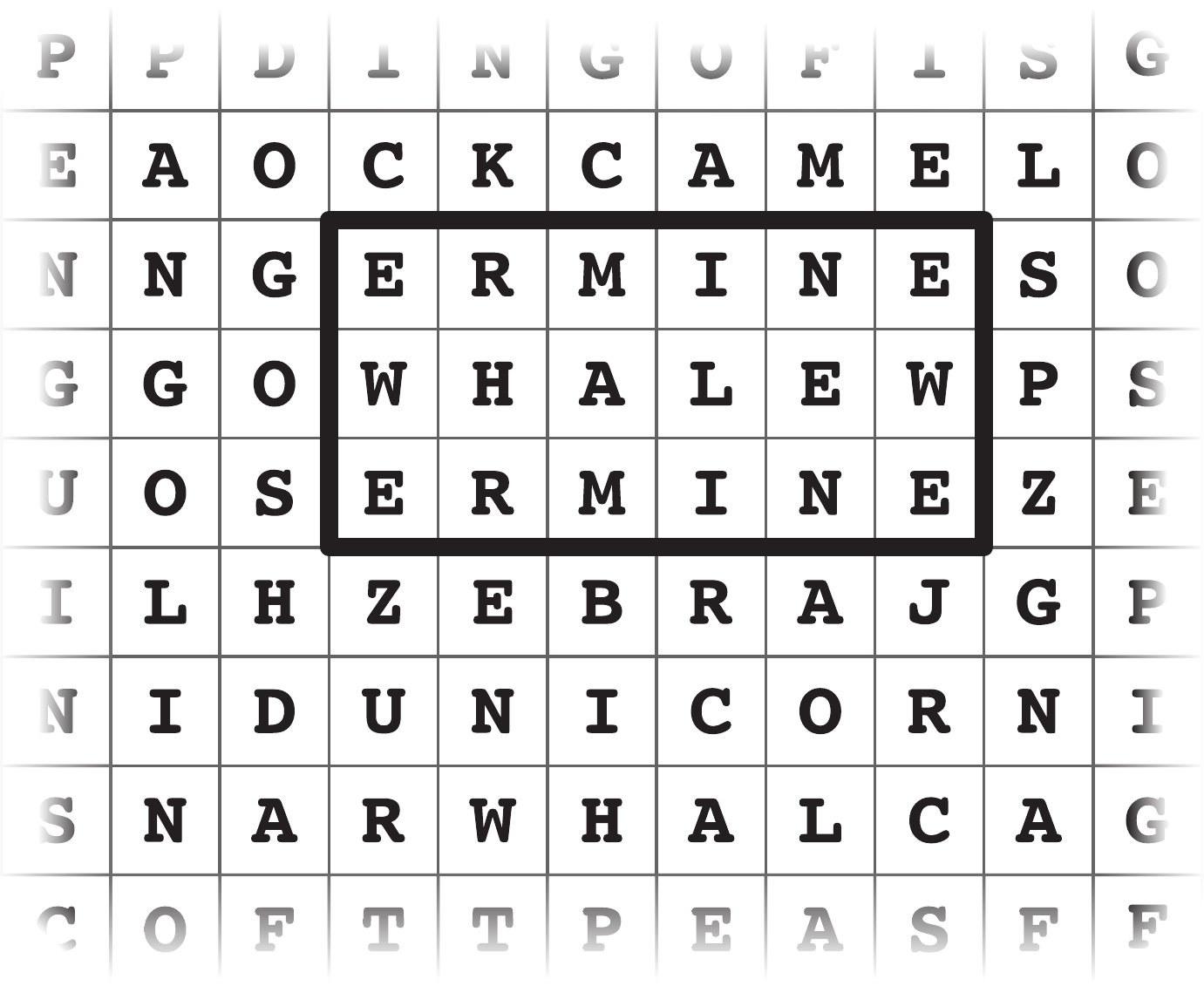}
\end{center}
\end{figure}

Now the following natural question occurs:  is there a $k$-coloring
of the plane lattice containing no picture frames at all?  We might
call such a coloring {\it frameless}.

At first, one might suspect the answer is no, since that is the case
for the one-dimensional analogue of picture frames.   The one-dimensional
analogue is a word of the form $axa$, where $a$ is a single letter, over
a $k$-letter alphabet.
By the pigeonhole principle, every word of length at least $k+1$
contains a repeated letter, and hence a word of the form $axa$.  So
there are no infinite words without $1$-dimensional picture frames.
Furthermore, the four corners of a two-dimensional picture frame
must all contain the same symbol, and we know that for any
finite alphabet, rectangles (in
fact, squares) of identical symbols must occur in any sufficiently
large region of the plane~\cite{Graham:1990}.

In this note we prove the following result:
\begin{theorem}
There exists a frameless $2$-coloring of the infinite plane lattice.
\label{main}
\end{theorem}

\section{The coloring and the proof}
\label{two}

We start by proving the result for the quarter plane $\Enn \times \Enn$.

\begin{theorem}
For $i, j \geq 0$ define the coloring $f$ by
$f(i,j) = t(i+j)$, where ${\bf t} = t(0) t(1) t(2)
\cdots$ is the Thue-Morse sequence.  Then $f$ is a frameless 
$2$-coloring of $\Enn \times \Enn$.
\label{main1}
\end{theorem}

Here are the first few rows and columns of this coloring, where
$0$ is denoted by a white square and $1$ by a black square.  The coloring
is easily seen to consist of shifted copies of the one-dimensional
Thue-Morse word.

\begin{figure}[H]
\begin{center}
\includegraphics[width=4in]{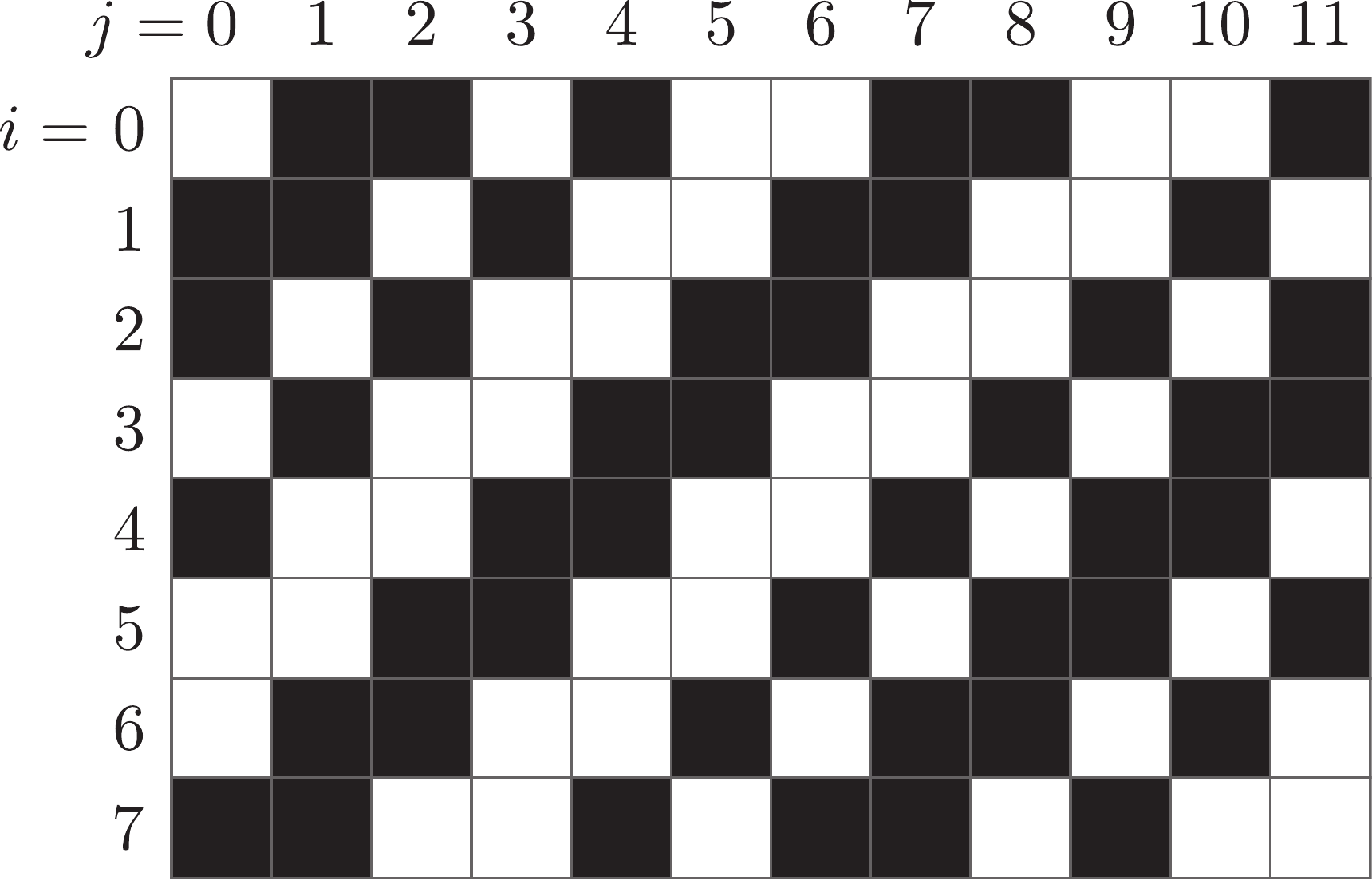}
\end{center}
\label{fig3}
\caption{Portion of a frameless 2-coloring}
\end{figure}

There are at least two ways to prove the result.   One uses
the theorem-proving software {\tt Walnut} again:
\begin{proof}
Let us write a first-order statement for the nonexistence of a
picture frame in $(f(i,j))_{i,j \geq 0}$:
\begin{multline*}
\neg ( \exists m,n,p,q (p \geq 1) \ \wedge \ (q \geq 1) \ \wedge \  \\
(\forall i \ (i \leq q) \implies f(m,n+i)=f(m+p,n+i)) \ \wedge \  \\
(\forall j \ (j \leq p) \implies f(m+j,n)=f(m+j,n+q)) .
\end{multline*}
Using the fact that $f(i,j)= t(i+j)$, this can be translated into
{\tt Walnut} as follows: \\
{\tt eval frameless "\char'176 ( E m,n,p,q (p>=1) \& (q>=1) \& }  \\
\centerline{\tt (Ai (i<=q) => T[m+n+i]=T[m+p+n+i]) 
\& (Aj (j<=p) => T[m+j+n]=T[m+j+n+q]))":} \\
Evaluating this statement in {\tt Walnut} 
gives the response {\tt true}, and requires
less than a second of CPU time on a laptop.
\end{proof}

Alternatively, we can prove Theorem~\ref{main1} directly.
We use the abbreviation
$f(a..b,c..d)$ to denote the rectangular block with rows from $a$ to $b$
and columns from $c$ to $d$. 
\begin{proof}
Suppose there exist $m, n, p, q$ such that $p, q \geq 1$
and $f(m,n..n+q) = f(m+p,n..n+q)$
and $f(m..m+p,n) = f(m..m+p,n+q)$.
Then, since $f(i,j) = t(i+j)$, we have
\begin{align*}
t(x+i) &= t(x+p+i), \ 0 \leq i \leq q; \\
t(x+j) &= t(x+q+j), \ 0 \leq j \leq p,
\end{align*}
where $x = m+n$.  Without loss of generality, assume $p \leq q$.
But then $t(x..x+2p)$ is an overlap, a contradiction.
\end{proof}

There's another way to view this result, using the notion of
{\it two-dimensional morphisms}.   These are morphisms that replace
each letter by a rectangular block of letters
\cite{Shallit&Stolfi:1989}.   If $a$ is
a $k \times \ell$ matrix, and $x$ is an $m \times n$ matrix, then
$h(a)$ is a $km \times \ell n$ matrix.

Now consider the two-dimensional morphism $\gamma$ defined as follows:
\begin{align*}
\gamma(0) &= \left[ \begin{array}{cc} 0 & 1 \\ 1 & 3 \end{array} \right] \\
\gamma(1) = \gamma(2) &= \left[ \begin{array}{cc} 3 & 0 \\ 0 & 3 \end{array} \right] \\
\gamma(3) &= \left[ \begin{array}{cc} 3 & 2 \\ 2 & 0 \end{array} \right] 
\end{align*}
and the coding defined by $\tau(i) = i \bmod 2$.
\begin{theorem}
$(f(i,j))_{i, j \geq 0}$ is given by $\tau(\gamma^\omega(0))$.
\end{theorem}

\begin{proof}
(Sketch)   The basic idea is that if we have the values of
$t(x)$ and $t(x+1)$, for $x = m+n$, then the values of
$f(2m+a,2n+b)$ for $a, b \in \{ 0,1\}$ can be computed as follows:
\begin{align*}
f(2m, 2n) &= t(x) \\
f(2m+1,2n) = f(2m, 2n+1) &= \overline{t(x)} \\
f(2m+1,2n+1) &= t(x+1),
\end{align*}
where the bar indicates binary complement ($\overline{0} = 1$ and
$\overline{1} = 0$).
\end{proof}

Let's look at the first few iterates, and their images under $\tau$:
$$ [0], \left[\begin{array}{cc} 0 & 1 \\ 1 & 3 \end{array}\right],
\left[\begin{array}{cccc} 
0 & 1& 3 & 0 \\
1 &3 &0 &3 \\
3 &0 &3 &2  \\
0 &3 &2 &0  
\end{array}\right], \ldots$$
$$ [0], \left[\begin{array}{cc} 0 & 1 \\ 1 & 1 \end{array}\right],
\left[\begin{array}{cccc} 
0 & 1& 1 & 0 \\
1 &1 &0 &1 \\
1 &0 &1 &0  \\
0 &1 &0 &0  
\end{array}\right]
$$
The arrays in the second line form larger and larger portions of the upper left corner of the array in Figure~\ref{fig3}.

\section{Extending our coloring to the whole plane}


Now that we have a coloring of $\Enn\times \Enn$, we can extend it to
$\Zee \times \Zee$ as follows, and prove this
explicit version of Theorem~\ref{main}.

\medskip

\noindent{\bf Theorem ${\bf 2}'$.}
{\it Using the extended definition of $t$ given in Eq.~\eqref{ext}, define
$f(i,j) = t(i+j)$ for $i, j \geq 0$.
Then $(f(i,j))_{i,j \in \Zee}$ is frameless.}

\begin{proof}
Exactly the same as the proof of Theorem~\ref{main1}, 
using the fact that the two-sided infinite word $\bf u$ is overlap-free.
\end{proof}

A slightly different example can be given using the matrix-valued morphism
$\gamma$ and the coding $\tau$ defined in Section~\ref{two}.
To do so, we start with the $2 \times 2$ array
$M := \gamma(0)$ and consider the center of the array to be a point from
which all four quarter-planes extend outward, as in the figure:
\begin{figure}[H]
\begin{center}
\includegraphics[width=2.5in]{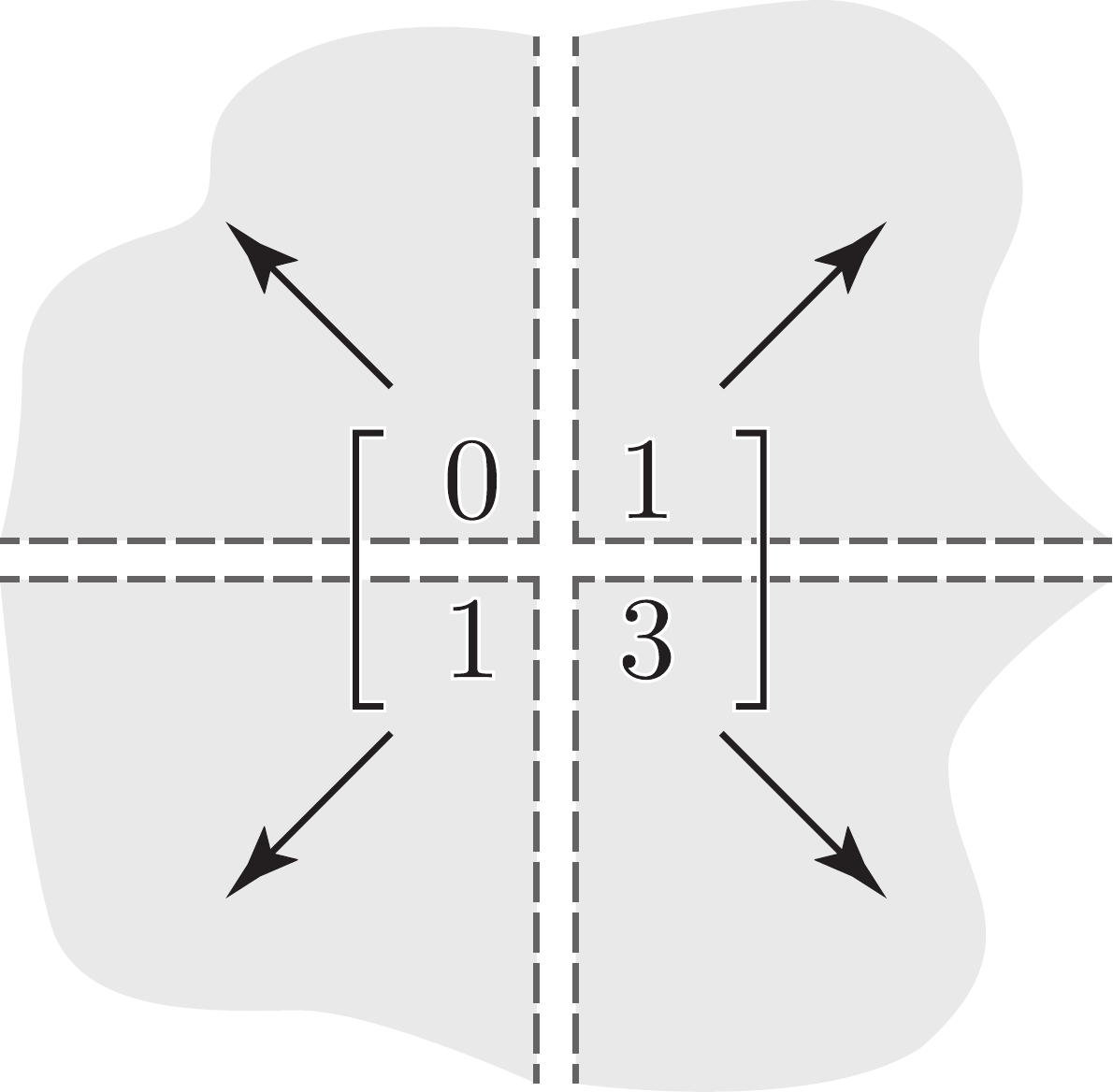}
\end{center}
\label{fig2}
\caption{Iterating a $2$-D morphism to cover the plane}
\end{figure}
We then want to iterate
$\gamma$ on each block.   In order for this procedure to work, we would
need the image of each letter $a$ to match the letter in the 
corresponding corner: 
\begin{itemize}
\item $0$ should appear in the lower right of its
image;
\item $1$ should appear in the upper right and lower left of its
image; and
\item $3$ should appear in the upper left of its image.
\end{itemize}

These properties do not hold for $\gamma$.  But they do
hold for $\gamma^2$ \ !:
\begin{align*}
\gamma^2(0) &= \left[ \begin{array}{cccc} 0 & 1 & 3 & 0 \\ 
	1 & 3 & 0 & 3 \\
	3 & 0 & 3 & 2 \\
	0 & 3 & 2 & 0
	\end{array} \right] &
\gamma^2(1) &= \left[ 
\begin{array}{cccc} 
3 & 2 & 0& 1 \\ 
2 & 0 & 1 & 3 \\
0 & 1 & 3 & 2 \\
1 & 3 & 2 & 0 
\end{array}
\right] &
\gamma^2(3) &= \left[
\begin{array}{cccc}
3 & 2 & 0 & 3 \\ 
2 & 0 & 3 & 0 \\
0 & 3 & 0 & 1 \\
3 & 0 & 1 & 3 \end{array}
\right] \ .
\end{align*}
Applying $\gamma^2$ to $M$ gives
\begin{center}
\begin{tabular}{c|c} 
$\begin{array}{cccc}
0 & 1 & 3 & 0 \\
1 & 3 & 0 & 3 \\
3 & 0 & 3 & 2 \\
0 & 3 & 2 & 0
\end{array}$  & 
$\begin{array}{cccc} 
3 & 2 & 0& 1 \\ 
2 & 0 & 1 & 3 \\
0 & 1 & 3 & 2 \\
1 & 3 & 2 & 0 
\end{array}$ \\
\hline
$\begin{array}{cccc}
3 & 2 & 0 & 1 \\
2 & 0 & 1 & 3 \\
0 & 1 & 3 & 2 \\
1 & 3 & 2 & 0
\end{array}$ &
$\begin{array}{cccc} 
3 & 2 & 0 & 3 \\
2 & 0 & 3 & 0 \\
0 & 3 & 0 & 1 \\
3 & 0 & 1 & 3 
\end{array} 
$
\end{tabular}
\end{center}
Hence, by iterating $\gamma$ infinitely, and
then applying $\tau$, we get a $2$-coloring of the
entire plane with the desired property.

\section{Connection to aperiodic tilings}

A {\it periodic tiling\/} is a tiling of the plane that has two
linearly independent translation symmetries. That is, there are 
two directions in which you can slide the tiling and have the tiles
line up perfectly with translated copies.  A set of shapes is called
{\it aperiodic\/} when they admit tilings of the plane, but none that
are periodic.  Non-periodic tilings are not necessarily interesting
of their own accord, because they might be constructed from shapes
that can trivially be rearranged into periodic tilings. The
interesting case is when the shapes themselves {\it force\/} long-range
aperiodicity, which is why aperiodicity is a property ascribed
to the shapes, known as {\it prototiles}. The most famous aperiodic
tile sets are those
discovered independently by Penrose and Ammann~\cite{GS}.

Consider a tiling whose prototiles are polyominoes (unit squares glued 
together along their edges).  Every member of the set appears in one of 
(at most) eight
rotated and reflected orientations in the tiling.  Assign the integers
$\{1,\ldots, n\}$ to the squares that make up all prototiles; then every
cell in a tiling from those prototiles can be given a label from
$\{1,\ldots, n\}\times\{1,\ldots 8\}$, describing the identity and orientation
of the prototile square covering that cell.  In other words, a tiling 
by polyominoes
is also a coloring of the plane lattice.  The figure shows two annotated
prototiles on the left, and a portion of a tiling/coloring on the right.
\begin{figure}[H]
\begin{center}
\includegraphics[width=4in]{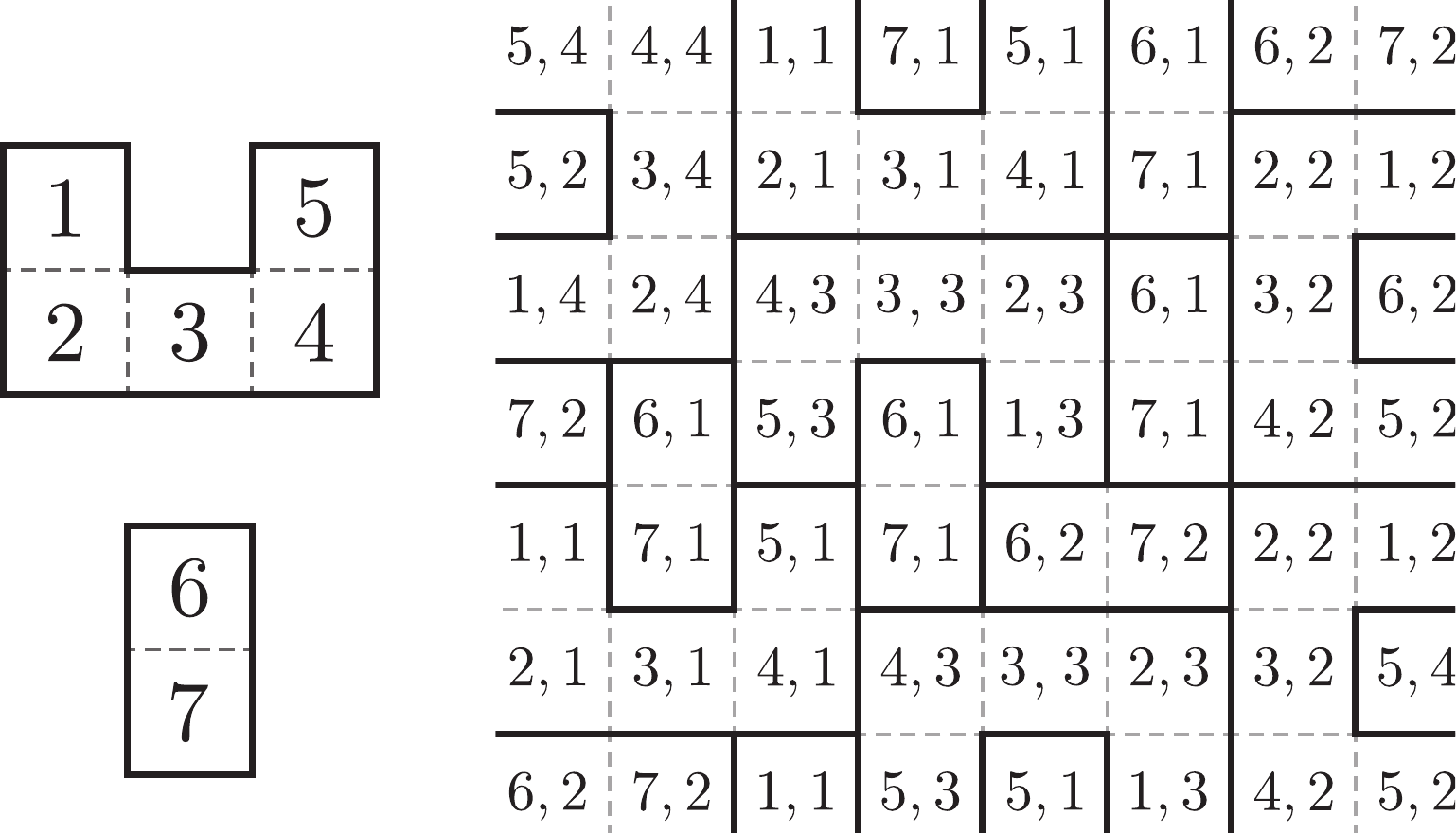}
\end{center}
\end{figure}
We know that there exist aperiodic sets of polyominoes---Winslow offers 
one example adapted from an earlier tiling by Ammann~\cite{Winslow}.
The existence of such prototiles suggests an indirect proof that frameless
colorings of the plane lattice must exist.  Consider a coloring constructed
by applying the labelling method above to a tiling by an aperiodic set of 
polyominoes.  If a frame existed in such a coloring, we would
be able to repeat it in a rectangular arrangement, with adjacent horizontal
and vertical copies overlapping by one column and one row, respectively.
The periodic coloring thus obtained would imply a periodic tiling by
the original polyominoes, contradicting the assumption that they were
aperiodic.

\newcommand{\noopsort}[1]{} \newcommand{\singleletter}[1]{#1}

\end{document}